\documentclass[11pt]{amsart}
\usepackage{amsmath,amssymb,amsthm,enumerate,a4wide,graphicx,mathtools}

\usepackage{xypic} 
\usepackage{mathrsfs} 
\usepackage{xspace}
\usepackage{etoolbox}
\usepackage{hyperref}

\usepackage[usenames, dvipsnames]{color}

\newtheorem{thm}{Theorem}[section]

\newtheorem{prop}[thm]{Proposition}
\newtheorem{lem}[thm]{Lemma}
\newtheorem{cor}[thm]{Corollary}

\theoremstyle{definition}
\newtheorem{defn}[thm]{Definition}

\newtheorem{rmk}[thm]{Remark}

\newtheorem*{claim*}{Claim}
\newtheorem{ack}[thm]{Acknowledgments}

\newcommand{\mc}[1]{\mathcal{#1}}

\newcommand{\ms}[1]{\mathscr{#1}} 

\newcommand{\Zb}{\mathbb{Z}}
\newcommand{\Nb}{\mathbb{N}}

\newcommand{\U}{\mathcal{U}}

\newcommand{\Es}{\ms{E}}

\newcommand{\tdlc}{t.d.l.c.\@\xspace}
\newcommand{\tdlcsc}{t.d.l.c.s.c.\@\xspace}
\newcommand{\scou}{s.c.\@\xspace}

\newcommand{\defbold}{\textbf}

\newcommand{\normal}{\trianglelefteq}

\newcommand{\ssl}{/\!\!/}
\newcommand{\inv}{^{-1}}
\newcommand{\triv}{\{1\}}

\newcommand{\Homeo}{\mathrm{Homeo}}

\newcommand{\Stab}{\mathrm{Stab}}

\newcommand{\grp}[1]{\langle #1 \rangle}
\newcommand{\ol}[1]{\overline{#1}}

\newcommand{\Sym}{\mathop{\rm Sym}\nolimits}

\newcommand{\Res}[1]{\mathop{\rm Res}_{#1}\nolimits}

\newcommand{\wt}[1]{\widetilde{#1}}
\newcommand{\wh}[1]{\widehat{#1}}

\begin{document}

\title[Homomorphisms into t.d.l.c. groups]{Homomorphisms into totally disconnected, locally compact groups with dense image}
\author{Colin D. Reid}
\address{  University of Newcastle,
   School of Mathematical and Physical Sciences,
   University Drive,
   Callaghan NSW 2308, Australia}
\email{colin@reidit.net}
\thanks{The first named author is an ARC DECRA fellow.  Research supported in part by ARC Discovery Project DP120100996.}
\author{Phillip R. Wesolek}
\address{Binghamton University,
Department of Mathematical Sciences,
PO Box 6000,
Binghamton New York 13902 USA}
\email{pwesolek@binghamton.edu}
\thanks{The second author was supported by ERC grant \#278469.}

\begin{abstract}
Let $\phi: G \rightarrow H$ be a group homomorphism such that $H$ is a totally disconnected locally compact (\tdlc) group and the image of $\phi$ is dense.  We show that all such homomorphisms arise as completions of $G$ with respect to uniformities of a particular kind.  Moreover, $H$ is determined up to a compact normal subgroup by the pair $(G,\phi\inv(L))$, where $L$ is a compact open subgroup of $H$.  These results generalize the well-known properties of profinite completions to the locally compact setting.  
\end{abstract}

\maketitle

\tableofcontents

\addtocontents{toc}{\protect\setcounter{tocdepth}{1}}

\section{Introduction}
For $G$ a (topological) group, the \defbold{profinite completion} $\hat{G}$ of $G$ is the inverse limit of the finite (continuous) quotients of $G$. One can obtain other profinite groups by forming the inverse limit of a suitable subset of the set of finite quotients of $G$. Such a profinite group $H$ is always a quotient of $\hat{G}$, and obviously the composition map $G\rightarrow \hat{G}\rightarrow H$ has dense image.

On the other hand, given a (topological) group $G$, one can ask which profinite groups $H$ admit a (continuous) homomorphism $\psi:G\rightarrow H$ with dense image. Letting $\iota:G\rightarrow \hat{G}$ be the canonical inclusion, it turns out that there is always a continuous quotient map $\wt{\psi}:\hat{G}\rightarrow H$ such that $\psi=\wt{\psi}\circ \iota$; cf. \cite[Lemma 3.2.1]{RZ00}. In this way one obtains a complete description of all profinite groups $H$ and homomorphisms $\psi:G\rightarrow H$ such that the image of $G$ is dense in $H$: all such groups and morphisms arise exactly by forming inverse limits of suitable sets of finite quotients of $G$. 

The aim of the present paper is to extend the well-known and well-established description of homomorphisms to profinite groups with dense image to homomorphisms into totally disconnected locally compact (\tdlc) groups with dense image. That is to say, we will develop a theory of \textit{\tdlc completions}. Using the language of uniformities, we shall see that this theory generalizes the profinite case.

Our approach generalizes previous work by Schlichting (\cite{Schlichting}) and Belyaev (\cite[\S7]{Belyaev}); Schlichting's completion has also been studied in subsequent work, e.g. \cite{SW_13}.  The novel contributions of this work are to present a unified theory of \tdlc completions and to identify properties that hold for every completion.

\subsection{Statement of results}
We shall state our results in the setting of Hausdorff topological groups. If one prefers, the group $G$ to be completed can be taken to be discrete. The topological group setting merely allows for finer control over the completions; for example, given a topological group $G$, there is often an interesting difference between completions of $G$ as a topological group and completions of $G$ as a discrete group.

\begin{defn}
For a topological group $G$, a \defbold{(\tdlc) completion map} is a continuous homomorphism $\psi:G\rightarrow H$ with dense image such that $H$ is a \tdlc group. We call $H$ a \defbold{\tdlc completion} of $G$.
\end{defn}

All \tdlc completions arise as completions with respect to a certain family of uniformities.
\begin{defn}
Let $G$ be a group and let $\mc{S}$ be a set of open subgroups of $G$.  We say that $\mc{S}$ is a \defbold{$G$-stable local filter} if the following conditions hold:
\begin{enumerate}[(a)]
\item $\mc{S}$ is non-empty;
\item Any two elements of $\mc{S}$ are commensurate;
\item Given a finite subset $\{V_1,\dots,V_n\}$ of $\mc{S}$, then $\bigcap^n_{i=1}V_i \in \mc{S}$, and given $V \le W \le G$ such that $|W:V|$ is finite, then $V \in \mc{S}$ implies $W \in \mc{S}$;
\item Given $V \in \mc{S}$ and $g \in G$, then $gVg\inv \in \mc{S}$.
\end{enumerate}
\end{defn}

Each $G$-stable local filter $\mc{S}$ is a basis at $1$ for a (not necessarily Hausdorff) group topology on $G$, and thus, there is an associated right uniformity $\Phi_r(\mc{S})$ on $G$.  The completion with respect to this uniformity, denoted by $\hat{G}_{\mc{S}}$, turns out to be a \tdlc group (Theorem~\ref{thm:completion:def}); we denote by $\beta_{G,\mc{S}}:G\rightarrow \hat{G}_{\mc{S}}$ the canonical inclusion homomorphism. All \tdlc completions moreover arise in this way.

\begin{thm}[see Theorem~\ref{thm:completion:characterization}]
If $G$ is a topological group and $\phi: G \rightarrow H$ is a \tdlc completion map, then there is a $G$-stable local filter $\mc{S}$ and a unique topological group isomorphism $\psi: \hat{G}_{\mc{S}} \rightarrow H$ such that $\phi = \psi \circ \beta_{(G,\mc{S})}$.
\end{thm}

We next consider completion maps $\phi: G \rightarrow H$ where a specified subgroup $U$ of $G$ is the preimage of a compact open subgroup of $H$.  In this case, there are two canonical completions of $G$ such that $U$ is the preimage of a compact open subgroup of the completion: the \defbold{Belyaev completion}, denoted by $\hat{G}_U $, and the \defbold{Schlichting completion}, denoted by $G \ssl U$. These completions are the `largest' and `smallest' completions in the following sense. We denote by $\beta_U:G\rightarrow \hat{G}_U$ and $\beta_{G/U}:G\rightarrow G\ssl U$ the canonical inclusion homomorphisms.

\begin{thm}[see Theorem~\ref{thm:completion:sandwich}]\label{thm:intro_sandwich}
Suppose that $G$ is a group and that $\phi: G \rightarrow H$ is a \tdlc completion map. Letting $U\leq G$ be the preimage of some compact open subgroup of $H$, then there are unique continuous quotient maps $\psi_1: \hat{G}_U \rightarrow H$ and $\psi_2: H \rightarrow G\ssl U$ with compact kernels such that the following diagram commutes:
\[
\xymatrixcolsep{3pc}\xymatrix{
& G\ar_{\beta_U}[ld] \ar^{\phi}[d]  \ar^{\beta_{G/U}}[rd]& \\
\hat{G}_U  \ar_{\psi_1}[r] & H \ar_{\psi_2}[r]& G\ssl U. }
\]
\end{thm}

We conclude by identifying several properties that are ``independent" of the completion. Doing so requires identifying a notion of size. Two subgroups $U$ and $V$ of a group $G$ are \defbold{commensurate} if $U\cap V$ has finite index in both $U$ and $V$. A subgroup $W\leq G$ is \textbf{commensurated} in $G$ if $W$ is commensurate with $gWg^{-1}$ for all $g\in G$. We write $[U]$ for the set of closed subgroups commensurate with $U$. We call the collection $[U]$ a \textbf{size} if some (equivalently, any) $W\in [U]$ is commensurated; for $[U]$ a size, observe that $[U] = [gUg\inv]$ for all $g \in G$.

A compact open subgroup $R$ of a \tdlc group $H$ is commensurated in $H$. Given a completion map $\psi:G\rightarrow H$, the preimage $U := \psi\inv(R)$ is an open subgroup of $G$ that is commensurated in $G$. We thus obtain a size $[U]$, and the size $[U]$ does not depend on the choice of $R$. We say that $[U]$ is the \defbold{size} of $\psi$.  

\begin{thm}[see \S\ref{sec:invariant}]\label{thmintro:invariant}
Let $G$ be a topological group.  For each of the following properties, either every completion of $G$ of size $\alpha$ has the property, or every completion of $G$ of size $\alpha$ fails to have the property. 
\begin{enumerate}[(1)]
\item Being $\sigma$-compact.
\item Being compactly generated.
\item Being amenable.
\item Being uniscalar.
\item Having a quotient isomorphic to $N$ where $N$ is any specified \tdlc group that has no non-trivial compact normal subgroups.
\end{enumerate}
\end{thm}

\begin{thm}[See \S\ref{sec:invariant}]
Let $G$ be a topological group. For each of the following properties, either every second countable completion of $G$ of size $\alpha$ has the property, or every second countable completion of $G$ of size $\alpha$ fails to have the property. 
\begin{enumerate}[(1)]
\item Being elementary.
\item Being elementary with transfinite rank $\beta$.
\end{enumerate}

\end{thm}

\begin{ack}
The first named author would like to thank Aleksander Iwanow for pointing out the article \cite{Belyaev} in response to an earlier preprint.
\end{ack}

\section{Preliminaries}

A quotient of a topological group must have closed kernel (such that the resulting quotient topology is Hausdorff). Topological group isomorphism is denoted by $\simeq$. We use ``t.d.", ``l.c.", and ``s.c." for ``totally disconnected", ``locally compact", and ``second countable", respectively. For a topological group $G$, the set $\U(G)$ is defined to be the collection of compact open subgroups of $G$.

\begin{defn}
A \defbold{Hecke pair} is a pair of groups $(G,U)$ where $G$ is a topological group and $U$ is an open subgroup of $G$ that is commensurated.
\end{defn}

\subsection{Uniformities}
Our approach to completions is via uniform spaces; our discussion of uniform spaces follows \cite{B_top1_89}.

\begin{defn}
Let $X$ be a set.  A \defbold{uniformity} is a set $\Phi$ of binary relations on $X$, called \defbold{entourages}, with the following properties:
\begin{enumerate}[(a)]
\item Each $A \in \Phi$ is reflexive, that is, $\{(x,x) \mid x \in X\} \subseteq A$.
\item For all $A,B \in \Phi$, there exists $C \in \Phi$ such that $C \subseteq A \cap B$.
\item For all $A \in \Phi$, there exists $B \in \Phi$ such that
\[
B \circ B := \{(x,z) \mid \exists y \in X: \{(x,y),(y,z)\} \subseteq B \}
\]
is a subset of $A$.
\item For all $A \in \Phi$, there exists $B \in \Phi$ such that the set $\{(y,x) \mid (x,y) \in B\}$ is a subset of $A$.
\end{enumerate}
\end{defn}

A set with a uniformity is called a \defbold{uniform space}.  Two uniformities $\Phi,\Phi'$ on a set $X$ are \defbold{equivalent} if every $A \in \Phi$ contains some $A' \in \Phi'$ and vice versa.

\begin{defn} Let $(X,\Phi)$ be a uniform space. A filter $f$ of subsets of $X$ is called a \defbold{minimal Cauchy filter} if $f$ is a $\subseteq$-least filter such that for all $U\in \Phi$ there is $A\in f$ with $A\times A\subseteq U$.
\end{defn}

A basis at $1$ of a topological group gives rise to two canonical uniformities:
\begin{defn} Suppose that $G$ is a topological group (not necessarily Hausdorff) and let $\mc{B}$ be a basis at $1$. The \defbold{left $\mc{B}$-uniformity} $\Phi_l(\mc{B})$ consists of entourages of the form
\[
U_l:=\{(x,y)\mid x^{-1}y\in U\} \quad (U \in \mc{B}).
\]
The \defbold{right $\mc{B}$-uniformity} $\Phi_r(\mc{B})$ consists of entourages of the form
\[
U_r:=\{(x,y)\mid xy^{-1}\in U\} \quad (U \in \mc{B}).
\]
\end{defn}

For both the left and right uniformities, the uniformity itself depends on the choice of basis, but the equivalence class of the uniformity does not. We will thus omit references to the basis where it is not significant.  (For definiteness, one can take $\mc{B}$ to be the set of all open identity neighborhoods, but it is often convenient to take another basis.) We will also use the definite article when referring to the left or the right uniformity.

\begin{defn}[{\cite[II.3.7]{B_top1_89}}] The \defbold{completion} of a uniform space $(X,\Phi)$ is defined to be 
\[
\hat{X}:=\{f\mid f\text{ is a minimal Cauchy filter}\}.
\]
along with the uniformity $\hat{\Phi}$ given by entourages of the form
\[
\hat{U}:=\{(f,g)\mid \exists A\in f\cap g\text{ with } A\times A\subseteq U\}
\]
where $U\in \Phi$.
\end{defn}
There is a canonical, continuous completion map $\beta:X\rightarrow \hat{X}$ which has dense image, defined by $x\mapsto f_x$ where $f_x$ is the minimal Cauchy filter containing the neighborhoods of $x$.  Note that as a topological space, $\hat{X}$ is determined by the equivalence class of $\Phi$.

In particular, if $G$ is a topological group equipped with the right uniformity, we let $\beta: G \rightarrow \hat{G}$ be the completion map associated to this uniformity.  Since $\beta$ has dense image, there is at most one way to equip $\hat{G}$ with a continuous group multiplication and inverse such that $\beta$ is a homomorphism; if these exist, we can say that $\hat{G}$ is a topological group in a canonical sense.  The completion $\hat{G}$ admits such a group multiplication exactly when the left and right uniformities are equivalent; equivalently, the inverse function preserves the set of minimal $\Phi_r$-Cauchy filters.

\begin{thm}[{\cite[III.3.4 Theorem 1]{B_top1_89}}]\label{thm:completions} Suppose that $G$ is a topological group and that $\Phi_r$ is the right uniformity. The completion $\hat{G}$ is a topological group if and only if the inverse map carries minimal $\Phi_r$-Cauchy filters to minimal $\Phi_r$-Cauchy filters.
\end{thm}

\begin{thm}[{\cite[III.3.4 Theorem 1]{B_top1_89}}] Suppose that $G$ is a topological group, $\Phi_r$ is the right uniformity, and the completion $\hat{G}$ is a topological group. Then the following hold:
\begin{enumerate}[(1)]
\item The map $\beta:G\rightarrow \hat{G}$ is a continuous homomorphism with dense image.
\item Multiplication on $\hat{G}$ is defined as follows: given $f,f' \in \hat{G}$, then $ff'$ is the minimal Cauchy filter of subsets of $G$ generated by sets $AB \subset G$ where $A \in f$ and $B \in f'$.
\end{enumerate}
\end{thm}

\section{A general construction for completions}\label{sec:general}

We describe a procedure for producing \tdlc completions.  The main idea in the construction is to form uniform completions of $G$ with respect to a family of group topologies that are in general coarser than the natural topology of $G$.

\begin{defn}
Let $G$ be a group and let $\mc{S}$ be a set of open subgroups of $G$.  We say that $\mc{S}$ is a \defbold{$G$-stable local filter} if the following conditions hold:
\begin{enumerate}[(a)]
\item $\mc{S}$ is non-empty;
\item Any two elements of $\mc{S}$ are commensurate;
\item $\mc{S}$ is a filter in its commensurability class: that is, given a finite subset $\{V_1,\dots,V_n\}$ of $\mc{S}$, then $\bigcap^n_{i=1}V_i \in \mc{S}$, and given $V \le W \le G$ such that $|W:V|$ is finite, then $V \in \mc{S}$ implies $W \in \mc{S}$;
\item $\mc{S}$ is stable under conjugation in $G$: that is, given $V \in \mc{S}$ and $g \in G$, then $gVg\inv \in \mc{S}$.
\end{enumerate}
We say $\mc{S}$ is a \defbold{$G$-stable local filter of size $[U]$} if in addition $\mc{S} \subseteq [U]$.
\end{defn}

\begin{rmk}
If $\mc{S}$ is a $G$-stable local filter, then $\mc{S}$ is a filter of $[V]$ for any $V\in \mc{S}$. Furthermore, $[V]$ must be stable under the conjugation action of $G$. 
\end{rmk}

\begin{lem}\label{lem:localfilter_topology}
Let $G$ be a group and let $\mc{S}$ be a $G$-stable local filter.  Then $\mc{S}$ is a basis at $1$ for a (not necessarily Hausdorff) group topology on $G$.
\end{lem}

\begin{proof}
Let $\mc{T}$ be the topology generated by all right translates of elements of $\mc{S}$.  Since $\mc{S}$ is invariant under conjugation in $G$, it is clear that every left coset of an element of $\mc{S}$ is a union of right cosets of elements of $\mc{S}$ and vice versa; hence $\mc{T}$ is invariant under inverses.  We see that the multiplication map $m: (g,h) \mapsto gh$ is continuous with respect to $\mc{T}$ by observing that given $U \in \mc{S}$ and $g,h \in G$, then $m\inv(Ugh)$ contains the open neighborhood $(Ug) \times (g\inv Ugh)$ of $(g,h)$.
\end{proof}

We remark that the largest quotient on which $\mc{S}$ induces a Hausdorff group topology is $G/K$ where $K$ is the normal subgroup $\bigcap_{U \in \mc{S}}U$.

Equipping $G$ with the group topology induced from $\mc{S}$, we can form the left and right uniformities $\Phi_l(\mc{S})$ and $\Phi_r(\mc{S})$.  

\begin{defn}
Let $\mc{S}$ be a $G$-stable local filter.  We define a \defbold{right $\mc{S}$-Cauchy filter} $f$ in $G$ to be a minimal Cauchy filter with respect to the uniformity $\Phi_r(\mc{S})$.  In other words, $f$ is a filter of subsets of $G$ with the following properties:
\begin{enumerate}[(a)]
\item For every $V \in \mc{S}$, there is exactly one right coset $Vg$ of $V$ in $G$ such that $Vg \in f$;
\item Every element of $f$ contains a right coset of some element of $\mc{S}$.
\end{enumerate}
\end{defn}

Left $\mc{S}$-Cauchy filters are defined similarly with respect to the left uniformity.  Notice that for each $g \in G$, there is a corresponding \defbold{principal} right $\mc{S}$-Cauchy filter $f_g$ generated by $\{Vg \mid V \in \mc{S}\}$. Where the choice of $\mc{S}$ is clear, we will write `Cauchy' to mean `$\mc{S}$-Cauchy'.

The next series of results will establish that the hypotheses of Theorem~\ref{thm:completions} are satisifed, so completing $G$ with respect to $\Phi_r(\mc{S})$ produces a topological group.

\begin{lem}\label{lem:leftright_cosets}Let $G$ be a group, $N$ be a commensurated subgroup of $G$, and $g \in G$.  Then there are $h_1,\dots,h_n\in G$ such that for all $h \in G$, the set $Ng \cap hN$ is a (possibly empty) union of finitely many right cosets of $N\cap \bigcap_{i=1}^nh_iNh_i^{-1}$.\end{lem}

\begin{proof} Suppose $Ng\cap hN\neq \emptyset$ and put $R: = N \cap g\inv Ng$. For all $h \in G$, we have $(Ng \cap hN)R = Ng \cap hN$, so $Ng \cap hN$ is a union of left cosets of $R$ in $G$.  The left cosets of $R$ in $G$ that are subsets of $Ng$ are exactly those of the form $gtR$ for $t \in g\inv Ng$; indeed,
\[
xR\subseteq Ng\Leftrightarrow g^{-1}xR\subseteq g^{-1}Ng\Leftrightarrow g^{-1}x\in g^{-1}Ng.
\]
Since $R$ has finite index in $g\inv Ng$, we deduce that only finitely many left cosets of the form $gtR$ exist. It follows that the set $\{Ng \cap hN \mid h \in G \text{ and } Ng\cap hN\neq \emptyset\}$ is finite. 

Let $h_1,\dots,h_n \in G$ satisfy
\[ 
\{Ng \cap hN \mid h \in G \text{ and } Ng\cap hN\neq \emptyset \} = \{Ng \cap h_1N, \dots, Ng \cap h_nN\}.
\]
Setting $M := N \cap \bigcap^n_{i=1} h_iNh\inv_i$, we see that $M(Ng \cap hN) = Ng \cap hN$ for all $h \in G$ with $Ng\cap hN\neq\emptyset$. Therefore, $Ng \cap hN$ is a union of right cosets of $M$. That this union is finite follows as in the previous paragraph. \end{proof}

\begin{lem}\label{lem:left_is_right}
Let $G$ be a topological group, $\mc{S}$ be a $G$-stable local filter, and $f$ be a set of subsets of $G$.  Then $f$ is a right $\mc{S}$-Cauchy filter in $G$ if and only if $f$ is a left $\mc{S}$-Cauchy filter in $G$.\end{lem}

\begin{proof}
By symmetry, it suffices to assume $f$ is a right Cauchy filter and prove that $f$ is a left Cauchy filter. 

Fixing $V \in \mc{S}$, the filter $f$ contains some right coset $Vg$ of $V$. Applying Lemma~\ref{lem:leftright_cosets} to $V$ and $g$, we produce a finite intersection $W$ of conjugates of $V$ such that $W \le V$ and such that for each $h \in G$, the set $Vg \cap hV$ is a union of right cosets of $W$. Since $\mc{S}$ is closed under conjugation and finite intersection, we additionally have $W \in \mc{S}$.

Since $f$ is right Cauchy, there is a unique right coset $Wk$ of $W$ contained in $f$, and since $\emptyset \not\in f$, it must be the case that $Wk \subseteq Vg$. Observing that $Vg=\bigcup_{h\in G}Vg\cap hV$, there is some $h\in G$ such that $Wk$ intersects $Vg \cap hV$. Lemma~\ref{lem:leftright_cosets} then ensures that indeed $Wk\subseteq Vg\cap hV$, hence $hV \in f$.  We conclude that $f$ contains a left coset of $V$ for every $V \in \mc{S}$.  Since $f$ is a filter and $\emptyset \not\in f$, in fact $f$ contains exactly one left coset of $V$ proving (a) of the definition of a left $\mc{S}$-Cauchy filter.

Given any element $A \in f$, then $A$ contains $Vg$ for some $V \in \mc{S}$ and $g \in G$; in particular, $A$ contains the left coset $g(g\inv Vg)$ of $g\inv Vg \in \mc{S}$, proving (b). We thus deduce that $f$ is also a left $\mc{S}$-Cauchy filter.
\end{proof}

\begin{cor}\label{cor:inverse_ok}
For $G$ a topological group and $\mc{S}$ a $G$-stable local filter, the set of right $\mc{S}$-Cauchy filters in $G$ is preserved by the map on subsets induced by taking the inverse.
\end{cor}

\begin{proof}
The map $x \mapsto x\inv$ induces a correspondence between the set of left Cauchy filters and the set of right Cauchy filters.  By Lemma~\ref{lem:left_is_right}, these two sets coincide, so the set of right Cauchy filters is invariant under taking inverses.
\end{proof}

We may now apply Theorem~\ref{thm:completions} to produce a completion of $G$ with respect to $\Phi_r(\mc{S})$, denoted $\hat{G}_{\mc{S}}$. Specifically,
\begin{itemize}
\item The elements of $\hat{G}_{\mc{S}}$ are the (right) $\mc{S}$-Cauchy filters in $G$.

\item The set $\hat{G}_{\mc{S}}$ is equipped with a uniformity with entourages of the form
\[
E_U := \{ (f,f') \mid \exists g \in G: Ug \in f \cap f'\} \quad (U \in \mc{S})
\]
and topology generated by this uniformity.

\item The map given by
\[
\beta_{(G,\mc{S})}: G \rightarrow \hat{G}_{\mc{S}}; \quad g \mapsto f_g
\]
is continuous, since the topology induced by $\Phi_r(\mc{S})$ is coarser than the topology on $G$.  The image is dense, and the kernel is $\bigcap\mc{S}$.

\item There is a unique continuous group multiplication on $\hat{G}_{\mc{S}}$ such that $\beta_{(G,\mc{S})}$ is a homomorphism.  In fact, we can define multiplication on $\hat{G}_{\mc{S}}$ as follows: given $f,f' \in \hat{G}_{\mc{S}}$, then $ff'$ is the minimal Cauchy filter of subsets of $G$ generated by sets $AB \subset G$ where $A \in f$ and $B \in f'$.
\end{itemize}
\begin{defn} For $G$ a topological group and $\mc{S}$ a $G$-stable local filter, we call $\hat{G}_{\mc{S}}$ the \defbold{completion} of $G$ with respect to $\mc{S}$.
\end{defn}

We now establish a correspondence between \tdlc completions and completions with respect to a $G$-stable local filter.

\begin{thm}\label{thm:completion:def} 
If $G$ is a group and $\mc{S}$ is a $G$-stable local filter, then the following hold:
\begin{enumerate}[(1)]
\item The topological group $\hat{G}_{\mc{S}}$ is a \tdlc completion of $G$.

\item There is a one-to-one correspondence between $\U(\hat{G}_{\mc{S}})$ and $\mc{S}$ given as follows: For a compact open subgroup $R$ of $\hat{G}_{\mc{S}}$, the subgroup $\beta\inv_{(G,\mc{S})}(R)$ is the corresponding element of $\mc{S}$, and for $V \in \mc{S}$, the subgroup $\ol{\beta_{(G,\mc{S})}(V)}$ is the corresponding compact open subgroup of $\hat{G}_{\mc{S}}$.
\item For $V \in \mc{S}$, the subgroup $\ol{\beta_{(G,\mc{S})}(V)}$ is naturally isomorphic as a topological group to the profinite completion of $V$ with respect to the quotients $V/N$ such that $N \in \mc{S}$ and $N \unlhd V$. 
\end{enumerate}
\end{thm}

\begin{proof}
For $V \in \mc{S}$, set $B_V := \{f \in \hat{G}_{\mc{S}} \mid V \in f\}$.  In view of the definition of the topology of $\hat{G}_{\mc{S}}$, the collection $\{B_{V} \mid V \in \mc{S}\}$ is a base of identity neighborhoods in $\hat{G}_{\mc{S}}$.  Moreover, each $B_V$ is closed under multiplication and inverse, so $B_V$ is a subgroup of $\hat{G}_{\mc{S}}$.  Therefore, $\hat{G}_{\mc{S}}$ has a base of identity neighborhoods consisting of open subgroups.

Fix $U \in \mc{S}$ and write $\beta:= \beta_{(G,\mc{S})}$.  The image $\beta(U)$ is a dense subgroup of $B_U$, so in fact $B_U = \ol{\beta(U)}$.  Define
\[
\mc{N}_U:=\left\{\bigcap_{u\in U}uVu^{-1}\mid U\geq V\in \mc{S} \right\} .
\]
The set $\mc{N}_U$ is precisely the set of elements of $\mc{S}$ that are normal in $U$, and these subgroups necessarily have finite index in $U$. Form $\hat{U}_\mc{S}$, the profinite completion of $U$ with respect to the finite quotients $\{U/N\mid N\in \mc{N}_U\}$. Representing $\hat{U}_\mc{S}$ as a closed subgroup of $ \prod_{N \in \mc{N}_U}U/N$ in the usual way, we define $\theta:B_U\rightarrow \hat{U}_\mc{S}$ by setting $\theta(f): = (Ng)_{N \in \mc{N}_U}$ where $Ng$ is the unique coset of $N$ in $f$. One verifies that $\theta$ is an isomorphism of topological groups, proving $(3)$.

The set $\{B_V \mid V \in \mc{S}\}$ is thus a basis at $1$ of compact open subgroups, so $\hat{G}_{\mc{S}}$ is a \tdlc group. Since $\beta$ has a dense image, $\hat{G}_{\mc{S}}$ is a \tdlc completion of $G$, verifying $(1)$.

Finally, we have a map $\eta: \mc{S} \rightarrow \U(\hat{G}_{\mc{S}})$ given by $\eta: V \mapsto B_V$.  We observe from the definition of $B_V$ that in fact $V = \beta\inv(B_V)$, so $\eta$ is injective, with inverse on $\U(\hat{G}_{\mc{S}}) \cap \eta(\mc{S})$ given by $R \mapsto \beta\inv(R)$.  Let $R$ be a compact open subgroup of $\hat{G}_{\mc{S}}$.  Since $\{B_{V} \mid V \in \mc{S}\}$ is a base of identity neighborhoods in $\hat{G}_{\mc{S}}$, there is some $V \in \mc{S}$ such that $B_V \le R$. The subgroup $B_V$ has finite index in $R$, since $R$ is compact and $B_V$ is open, so $\beta\inv(B_V) = V$ has finite index in $\beta\inv(R)$.  By the definition of an $\mc{S}$-stable local filter, we therefore have $\beta\inv(R) \in \mc{S}$.  We conclude that $\eta$ is a bijection of the form required for $(2)$.
\end{proof}

For $V\in \mc{S}$, we will often abuse notation slightly and say that the profinite completion $\hat{V}_{\mc{S}}$ equals the compact open subgroup $\ol{\beta_{(G,\mc{S})}(V)}$ of $\hat{G}_{\mc{S}}$.  We will also write $\beta$ in place of $\beta_{(G,\mc{S})}$ when $G$ and $\mc{S}$ are clear from the context.

\section{Classification and factorization of completions}

Theorem~\ref{thm:completion:def} gives a method for producing \tdlc completions of a group $G$. In fact, just as in the profinite case, we see that \textit{all} \tdlc completions of $G$ arise in this way.  We first prove a criterion for whether a homomorphism from $G$ to a \tdlc group factors through $\hat{G}_{\mc{S}}$.

\begin{prop}\label{prop:filter_image}
Let $\phi: G \rightarrow H$ be a continuous homomorphism such that $H$ is a \tdlc group and let $\mc{S}$ be a $G$-stable local filter.  Then the following are equivalent:
\begin{enumerate}[(1)]
\item There is a continuous homomorphism $\psi: \hat{G}_{\mc{S}} \rightarrow H$ such that $\phi = \psi \circ \beta_{(G,\mc{S})}$.
\item Every open subgroup of $H$ contains $\phi(V)$ for some $V \in \mc{S}$.
\end{enumerate}
Moreover, if $(1)$ holds, then $\psi$ is unique.
\end{prop}

\begin{proof}
Suppose that $(1)$ holds and let $U$ be an open subgroup of $H$.  The preimage $\psi\inv(U)$ is an open subgroup of $\hat{G}_{\mc{S}}$, so $R \le \psi\inv(U)$ for some compact open subgroup $R$ of $\hat{G}_{\mc{S}}$.  Theorem~\ref{thm:completion:def} ensures that $V:= \beta\inv(R) \in \mc{S}$, and $V $ is such that $\phi(V) = \psi(\beta(V)) \le U$.  We conclude $(2)$.  Additionally, since $\beta(G)$ is dense in $\hat{G}_{\mc{S}}$, the equation $\phi = \psi \circ \phi$ determines the restriction of $\psi$ to $\beta(G)$ and hence determines $\psi$ uniquely as a continuous map.

Conversely, suppose that every open subgroup of $H$ contains $\phi(V)$ for some $V \in \mc{S}$.  For $f \in \hat{G}_{\mc{S}}$, define
\[ 
\hat{f}: = \bigcap \{\overline{\phi(Vg)} \mid g \in G, V \in \mc{S},\text{ and } Vg \in f\}.
\]
Since $H$ is a \tdlc group, the open subgroups of $H$ form a basis of identity neighborhoods.  By $(1)$, it follows that the set $\{\overline{\phi(V)} \mid V \in \mc{S}\}$ contains arbitrarily small subgroups of $H$, so its intersection $\hat{1}$ is the trivial group.

For general $f\in \hat{G}_{\mc{S}}$, $|\hat{f}|\le 1$ since $\hat{f}$ is an intersection of cosets of arbitrarily small subgroups of $H$.  Fix $R \in \U(H)$ and $Q\in \mc{S}$ such that $\ol{\phi(Q)} \le R$, so in particular $\ol{\phi(Q)}$ is compact. Fix $h\in G$ such that $Qh \in f$. Letting $g \in G$ and $V \in \mc{S}$ be such that $Vg \in f$, there is some $k \in G$ such that $(V \cap Q)k \in f$. The collection $f$ is a proper filter, so we see that $(V \cap Q)k \subseteq Vg \cap Qh$.  In particular, $\overline{\phi(V \cap Q)k}$ is contained in $\overline{\phi(Vg)}$. We can thus write $\hat{f}$ as $\bigcap_{C \in \mc{C}} C$ where
\[
\mc{C} = \{\overline{\phi(Vg)} \mid g \in G, V \in \mc{S}, Vg \subseteq Qh \text{ and } Vg \in f\}.
\]
For any $Q_1g_1,\dots,Q_ng_n \in f$ such that $Q_ig_i\subseteq Qh$, the intersection $\bigcap^n_{i=0} Q_ig_i$ is non-empty as it is an element of $f$. Thus $\mc{C}$ is a family of closed subsets of the compact set $\ol{\phi(Qg)}$ with the finite intersection property. We conclude that $\hat{f}$ is nonempty, so $|\hat{f}| =1$.  We now define a function $\psi: \hat{G}_{\mc{S}}\rightarrow H$ by setting $\psi(f)$ to be the unique element of $\hat{f}$. One verifies that $\psi$ is a homomorphism satisfying $\phi = \psi \circ \beta$.

To see that $\psi$ is continuous, fix $(V_{\alpha})_{\alpha\in I}$ a basis at $1$ of compact open subgroups for $H$ and for each $\alpha \in I$, choose $W_{\alpha} \in \mc{S}$ such that $\phi(W_{\alpha}) \le V_{\alpha}$.  As in the previous paragraph, we observe that $\bigcap _{\alpha\in I}\ol{\phi(W_{\alpha})}=\{1\}$. Consider a convergent net $f_{\delta}\rightarrow f$ in $\hat{G}_{\mc{S}}$. For each subgroup $W_{\alpha}$, there is $\eta_{\alpha}\in I$ such that $f_{\gamma}f^{-1}_{\gamma'}$ contains $W_{\alpha}$ for all $\gamma,\gamma'\geq \eta_{\alpha}$. We conclude that $\psi(f_{\gamma}f^{-1}_{\gamma'})\in \ol{\phi(W_{\alpha})} \le V_{\alpha}$ for all such $\gamma,\gamma'$.  In other words, $\psi(f_{\delta})$ is a Cauchy net in $H$ with respect to the right uniformity of $H$.  Since $H$ is locally compact, it is complete with respect to this uniformity, and $\psi(f_{\delta})$ converges. It now follows that $\psi(f) = \lim_{\delta\in I}\psi(f_{\delta})$, so $\psi$ is continuous as claimed, completing the proof that $(2)$ implies $(1)$.
\end{proof}

As a corollary, we note the case of Proposition~\ref{prop:filter_image} when $H$ is itself the completion with respect to a $G$-stable local filter.

\begin{cor}\label{cor:completion:factorization}
Let $G$ be a group with $G$-stable local filters $\mc{S}_1$ and $\mc{S}_2$.  Then the following are equivalent:
\begin{enumerate}[(1)]
\item There is a continuous homomorphism $\psi:\hat{G}_{\mc{S}_1}\rightarrow \hat{G}_{\mc{S}_2}$ such that $\beta_{(G,\mc{S}_2)} = \psi \circ \beta_{(G,\mc{S}_1)}$.
\item For all $V_2 \in \mc{S}_2$, there exists $V_1 \in \mc{S}_1$ such that $V_1 \le V_2$.
\end{enumerate}

\end{cor}

\begin{proof}
Set $\beta_1:=\beta_{(G,\mc{S}_1)}$ and $\beta_2:=\beta_{(G,\mc{S}_2)}$.

Suppose that $(1)$ holds and take $V_2 \in \mc{S}_2$.  There exists $L \in \U(\hat{G}_{\mc{S}_2})$ such that $V_2 = \beta\inv_2(L)$.  Applying Proposition~\ref{prop:filter_image} with the filter $\mc{S}_1$ gives $V_1 \in \mc{S}_1$ such that $\beta_2(V_1) \le L$. We deduce that $V_1 \le V_2$, verifying $(2)$.

Conversely, suppose that $(2)$ holds. For every open subgroup $U$ of $\hat{G}_{\mc{S}_2}$, we have $\beta_2(V_2) \le U$ for some $V_2 \in \mc{S}_2$ by Theorem~\ref{thm:completion:def}(2).  By hypothesis, there exists $V_1 \in \mc{S}_1$ such that $V_1 \le V_2$, so $\beta_2(V_1) \le U$.  The conclusion $(1)$ now follows by Proposition~\ref{prop:filter_image}.
\end{proof}

We will now demonstrate that all \tdlc completions arise as completions with respect to $G$-stable local filters.

\begin{thm}\label{thm:completion:characterization} 
If $G$ is a group and $H$ is a \tdlc completion of $G$ via $\phi: G \rightarrow H$, then the set of all preimages of compact open subgroups of $H$ is a $G$-stable local filter $\mc{S}$, and moreover, there is a unique topological group isomorphism $\psi: \hat{G}_{\mc{S}} \rightarrow H$ such that $\phi = \psi \circ \beta_{(G,\mc{S})}$.
\end{thm}
\begin{proof}
Setting $\mc{S} := \{ \phi\inv(M) \mid M \in \U(H)\}$, one verifies that $\mc{S}$ is a $G$-stable local filter.  Define $\psi: \hat{G}_{\mc{S}} \rightarrow H$ as in Proposition~\ref{prop:filter_image}, so $\psi$ is the unique continuous homomorphism such that $\phi = \psi \circ \beta$.  It is easily verified that $\psi$ is injective.

The group $\hat{G}_{\mc{S}}$ has a base of identity neighborhoods of the form $\{\hat{V}_{\mc{S}} \mid V \in \mc{S}\}$. Given $V \in \mc{S}$, say that $\phi\inv (P)=V$ for $P\in \U(H)$. The image $\phi(V)$ is dense in $P$, and the image $\beta(V)$ is dense in $\hat{V}_{\mc{S}}$. Since $\phi = \psi \circ \beta$, we infer that $\psi(\hat{V}_{\mc{S}})$ is also a dense subgroup of $P$.  The map $\psi$ is continuous and $\hat{V}_{\mc{S}}$ is compact, so in fact $\psi(\hat{V}_{\mc{S}}) = P$.  Hence, $\psi$ is an open map.  Since the image of $\psi$ is dense, it follows that $\psi$ is surjective and thereby bijective.

The map $\psi$ is an open continuous bijective homomorphism.  We conclude that $\psi$ is an isomorphism of topological groups, completing the proof.
\end{proof}
 
\begin{rmk}Theorem~\ref{thm:completion:characterization} shows that, up to isomorphism, all \tdlc completions of a group $G$ have the form $\hat{G}_{\mc{S}}$ for $\mc{S}$ some $G$-stable local filter.  This applies to the particular case where the \tdlc completion is a profinite group. For instance, the profinite completion of a discrete group $G$ is $\hat{G}_{\mc{S}}$ where $\mc{S}$ is the set of all subgroups of finite index in $G$. 
\end{rmk}

We conclude this section by characterizing properties of the kernel of the homomorphism $\psi$ obtained in Corollary~\ref{cor:completion:factorization} in terms of the $G$-stable local filters.  In light of Theorem~\ref{thm:completion:characterization}, these characterizations in fact apply to any continuous homomorphism $\psi: H_1 \rightarrow H_2$ where $H_1$ and $H_2$ are completions of $G$, via taking $\mc{S}_i$ to be the set of preimages of the compact open subgroups of $H_i$.

\begin{prop}\label{prop:completions:comparison}
Let $G$ be a group with $G$-stable local filters $\mc{S}_1$ and $\mc{S}_2$ such that for all $V_2\in \mc{S}_2$ there is $V_1\in \mc{S}_1$ with $V_1\leq V_2$.  Let $\psi:\hat{G}_{\mc{S}_1}\rightarrow \hat{G}_{\mc{S}_2}$ be the unique continuous homomorphism such that $\beta_{(G,\mc{S}_2)} = \psi \circ \beta_{(G,\mc{S}_1)}$, as given by Corollary~\ref{cor:completion:factorization}.  Then the following holds:
\begin{enumerate}[(1)]
\item The kernel of $\psi$ is compact if and only if $\mc{S}_2 \subseteq \mc{S}_1$.  If $\mc{S}_2 \subseteq \mc{S}_1$, then $\psi$ is also a quotient map.
\item The kernel of $\psi$ is discrete if and only if there exists $V_1 \in \mc{S}_1$ such that for all $ W_1 \in \mc{S}_1$ there is $W_2 \in \mc{S}_2$ with $V_1\cap W_2\le W_1$. 
\item The kernel of $\psi$ is trivial if and only if for all $ V_1 \in \mc{S}_1$,
$$V_1 = \bigcap \{V_2 \in \mc{S}_2 \mid V_1 \le V_2\}.$$
\end{enumerate}
\end{prop}
\begin{proof}
Set $\beta_1:=\beta_{(G,\mc{S}_1)}$ and $\beta_2:=\beta_{(G,\mc{S}_2)}$.

\textbf{Proof of $(1)$.}
Suppose that $\ker\psi$ is compact.  Taking $V \in \mc{S}_2$, the group $\psi\inv(\hat{V}_{\mc{S}_2})$ is a compact open subgroup of $\hat{G}_{\mc{S}}$ such that the preimage under $\beta_1$ is $V$.  Given the characterization of compact open subgroups of $\hat{G}_{\mc{S}_1}$ established in Theorem~\ref{thm:completion:def}(2), it follows that $V \in \mc{S}_1$.  Hence, $\mc{S}_2 \subseteq \mc{S}_1$.

Conversely, suppose that $\mc{S}_2 \subseteq \mc{S}_1$ and let $V \in \mc{S}_2$.  Given the construction of $\psi$ in the proof of Proposition~\ref{prop:filter_image}, we see that $\ker\psi \le \hat{V}_{\mc{S}_1}$; in particular, $\ker\psi$ is compact.  We see from Theorem~\ref{thm:completion:def} that $\psi(\hat{V}_{\mc{S}_1})= \ol{\beta_2(V)} = \hat{V}_{\mc{S}_2}$, so the image of $\psi$ is open.  Since the image of $\psi$ is also dense, it follows that $\psi$ is surjective.  We conclude that $\psi$ is a quotient map with compact kernel, as required.

\

\textbf{Proof of $(2)$.}
Suppose that $\ker\psi$ is discrete, so there is $V_1 \in \mc{S}_1$ such that $\ker\psi \cap (\hat{V_1})_{\mc{S}_1} = \triv$.  The map $\psi$ restricts to a topological group isomorphism from $(\hat{V_1})_{\mc{S}_1}$ to its image.  Take $W_1 \in \mc{S}_1$ and set $Y := V_1 \cap W_1$.  The subgroup $(\hat{Y})_{\mc{S}_1}$ is open in $(\hat{V_1})_{\mc{S}_1}$, so $\psi((\hat{Y})_{\mc{S}_1})$ is open in $\psi((\hat{V_1})_{\mc{S}_1})$.  There thus exists $W_2 \in \mc{S}_2$ such that $\psi((\hat{V_1})_{\mc{S}_1}) \cap (\hat{W_2})_{\mc{S}_2}  \le \psi((\hat{Y})_{\mc{S}_1})$.  Since $\psi$ is injective on $(\hat{V_1})_{\mc{S}_1}$, it follows that 
\[
(\hat{V_1})_{\mc{S}_1} \cap \psi\inv((\hat{W_2})_{\mc{S}_2}) \le (\hat{Y})_{\mc{S}_1}.
\]
We deduce that $V_1 \cap W_2 \le Y \le W_1$ since $\psi\circ \beta_1=\beta_2$.

\

Conversely, suppose that $V_1 \in \mc{S}_1$ is such that for all $ W_1 \in \mc{S}_1$ there is $W_2 \in \mc{S}_2$ with $V_1\cap W_2\le W_1$.  We claim that $(\hat{V_1})_{\mc{S}_1}$ intersects $\ker\psi$ trivially, which will demonstrate that $\ker\psi$ is discrete.  Suppose that $f \in (\hat{V_1})_{\mc{S}_1} \cap \ker\psi$ and suppose for contradiction that $f$ is non-trivial. There is then some $W_1 \in \mc{S}_1$ such that $f$ contains a nontrivial coset $W_1g$ of $W_1$; we may assume that $W_1 \le V_1$.  Since $f \in (\hat{V_1})_{\mc{S}_1}$, we additionally have $g \in V_1$.  

By the hypotheses, there is $W_2 \in \mc{S}_2$ such that $V_1 \cap W_2 \le W_1$, and there is $Y \in \mc{S}_1$ such that $Y\le V_1 \cap W_2$. Letting $g'\in G$ be such that $Yg'\in f$, it must be the case that $Yg' \subseteq W_2$ since $\psi(f)$ is trivial. On the other hand, $Yg'\subseteq W_1g$, so $g' \in V_1$. We now see that $g' \in V_1 \cap W_2 \le W_1$ and that $W_1g=W_1g' = W_1$. Thus $W_1g$ is the trivial coset of $W_1$, a contradiction.  We conclude that $(\hat{V_1})_{\mc{S}_1} \cap \ker\psi = \triv$ as claimed.

\

\textbf{Proof of $(3)$.}
Suppose that $\psi$ is injective.  Let $V_1 \in \mc{S}_1$ and set $\mc{R}: = \{V_2 \in \mc{S}_2 \mid V_1 \le V_2\}$.  The subgroup $(\hat{V_1})_{\mc{S}_1}$ is a compact open subgroup of $\hat{G}_{\mc{S}_1}$, so $K: = \psi((\hat{V_1})_{\mc{S}_1})$ is a compact, hence profinite, subgroup of $\hat{G}_{\mc{S}_2}$.  It follows that $K$ is the intersection of all compact open subgroups of $\hat{G}_{\mc{S}_2}$ that contain $K$. 

All compact open subgroups of $\hat{G}_{\mc{S}_2}$ are of the form $\hat{W}_{\mc{S}_2}$ for some $W\in \mc{S}_2$. For $\hat{W}_{\mc{S}_2}\ge K$ compact and open, we thus have $\beta^{-1}_2(K)\leq W$, and since $\beta_1=\psi\circ \beta_2$, we deduce further that $V_1=\beta^{-1}_1((\hat{V}_1)_{\mc{S}_1})\leq W$. Hence, $K = \bigcap \{ \hat{W}_{\mc{S}_2} \mid W\in \mc{R}\}$.

For $g \in \bigcap \mc{R}$, it is the case that $\beta_2(g) \in \beta_2(W)$ for all $W \in \mc{R}$, so $\beta_2(g) \in K$.  In other words,
\[
\psi\circ\beta_1(g) \in \psi((\hat{V_1})_{\mc{S}_1}).
\]
As $\psi$ is injective, $\beta_1(g) \in (\hat{V_1})_{\mc{S}_1}$, and thus, $g \in \beta\inv_1((\hat{V_1})_{\mc{S}_1}) = V_1$, showing that $\bigcap \mc{R} \subseteq V_1$.  On the other hand it is clear from the definition of $\mc{R}$ that $V_1 \subseteq \bigcap\mc{R}$, so equality holds as required.

\

Conversely, suppose that for all $V_1 \in \mc{S}_1$, it is the case that $V_1 = \bigcap \{V_2 \in \mc{S}_2 \mid V_1 \le V_2\}$.  Fix $f \in \ker\psi$.  For each $V_1 \in \mc{S}_1$, we may write $f = \beta_1(g)u$ for $g \in G$ and $u \in (\hat{V}_1)_{\mc{S}_1}$, since $\beta_1(G)$ is dense in $\hat{G}_{\mc{S}_1}$, and it follows that $\beta_2(g) \in \psi((\hat{V}_1)_{\mc{S}_1})$.  We now infer that $g \in V_2$ for all $V_2 \in \mc{S}_2$ such that $V_2 \ge V_1$; by our hypothesis, it follows that $g \in V_1$. Recalling that $f=\beta_1(g)u$, it follows that $f\in (\hat{V}_1)_{\mc{S}_1}$. The subgroups $(\hat{V}_1)_{\mc{S}_1}$ form a basis at $1$, so indeed $f=1$. We conclude that $\psi$ is injective as required.
\end{proof}

\section{Canonical completions of Hecke pairs}

We now consider the \tdlc completions $H$ of $G$ such that a specified subgroup $U$ of $G$ is the preimage of a compact open subgroup of the completion.  The pair $(G,U)$ is then a Hecke pair.

Let us make several definitions to organize our discussion of \tdlc completions.  Recall that all preimages of compact open subgroups of a given range group $H$ are commensurate, so the following definition does not depend on the choice of $U$. 

\begin{defn}
Given a \tdlc completion map $\phi: G \rightarrow H$, the \defbold{size} of $\phi$ is defined to be $[U]$, where $U$ is the preimage of a compact open subgroup of $H$.  We also say that $[U]$ is the size of $H$ when the choice of completion map is not important, or can be inferred from the context.

A \defbold{completion map} for a Hecke pair $(G,U)$ is a continuous homomorphism $\phi: G \rightarrow H$ with dense image such that $H$ is a \tdlc group and $U$ is the preimage of a compact open subgroup of $H$. We say that $H$ is a \defbold{completion} of $(G,U)$.  When $H$ is also second countable, we call $H$ a \defbold{second countable completion}.
\end{defn}

Given a Hecke pair $(G,U)$, there are two canonical $G$-stable local filters containing $U$, defined as follows: The \defbold{Belyaev filter} is $[U]$.  The \defbold{Schlichting filter} $\mc{S}_{G/U}$ for $(G,U)$ is the filter of $[U]$ generated by the conjugacy class of $U$ -- that is,
\[
\mc{S}_{G/U}:=\left\{V\in [U]\mid \exists g_1,\dots,g_n\in G\text{ such that }\bigcap_{i=1}^ng_iUg_i^{-1}\leq V\right\}.
\]

\begin{defn}
The \defbold{Belyaev completion} for $(G,U)$, denoted by $\hat{G}_U$, is defined to be $\hat{G}_{[U]}$. The canonical inclusion map $\beta_{(G,[U])}$ is denoted by $\beta_{U}$. The \defbold{Schlichting completion} for $(G,U)$, denoted by $G\ssl U$, is defined to be $\hat{G}_{\mc{S}_{G/U}}$. The canonical inclusion map $\beta_{(G,\mc{S}_{G/U})}$ is denoted by $\beta_{G/U}$.
\end{defn}

\begin{rmk}
We stress that the commensurability class $[U]$ does not determine the Schlichting filter $\mc{S}_{G/U}$; indeed, the only situation when there is only one Schlichting filter of a given size is when the Schlichting filter is equal to the Belyaev filter of that size. 
\end{rmk}

Given any $G$-stable local filter $\mc{S}$ of size $[U]$ that contains $U$, we have $\mc{S}_{G/U}\subseteq \mc{S}\subseteq [U]$. Amongst $G$-stable local filters that contain $U$, the Schlichting filter is minimal whilst $[U]$ is maximal. The Belyaev and Schlichting completions are thus maximal and minimal completions of a Hecke pair in the following strong sense:

\begin{thm}\label{thm:completion:sandwich}
Suppose that $G$ is a group and that $\phi: G \rightarrow H$ is a completion map. Letting $U\leq G$ be the preimage of some compact open subgroup of $H$, then $(G,U)$ is a Hecke pair, $H$ is a completion of $(G,U)$, and there are unique continuous quotient maps $\psi_1: \hat{G}_U \rightarrow H$ and $\psi_2: H \rightarrow G\ssl U$ with compact kernels such that the following diagram commutes:
\[
\xymatrixcolsep{3pc}\xymatrix{
& G\ar_{\beta_U}[ld] \ar^{\phi}[d]  \ar^{\beta_{G/U}}[rd]& \\
\hat{G}_U  \ar_{\psi_1}[r] & H \ar_{\psi_2}[r]& G\ssl U. }
\]
\end{thm}

\begin{proof}
Fix $L$ a compact open subgroup of $H$ and set $U: = \phi\inv(L)$.  The subgroup $U$ is an open commensurated subgroup of $G$, so $(G,U)$ is a Hecke pair.  Since $\phi$ has dense image, we conclude that $H$ is a completion of $(G,U)$.  Set $\mc{S}:=\{\phi^{-1}(V)\mid V\in \U(H)\}$. By Theorem~\ref{thm:completion:characterization}, $\mc{S}$ is a $G$-stable local filter, and there is a unique topological group isomorphism $\psi:\hat{G}_{\mc{S}}\rightarrow H$ such that $\phi=\psi\circ\beta_{(G,\mc{S})}$.  Observe that $U \in \mc{S}$ by Theorem~\ref{thm:completion:def}(2).

Since $\mc{S}$ contains $U$, Proposition~\ref{prop:completions:comparison} ensures there are unique continuous quotient maps $\pi_1$ and $\pi_2$ with compact kernels such that the following diagram commutes:
\[
\xymatrixcolsep{3pc}\xymatrix{
& G\ar_{\beta_U}[ld] \ar^{\beta_{(G,\mc{S})}}[d]  \ar^{\beta_{G/U}}[rd]& \\
\hat{G}_U  \ar_{\pi_1}[r] & \hat{G}_{\mc{S}} \ar_{\pi_2}[r]& G\ssl U. }
\]
It follows that $\psi_1:=\psi\circ \pi_1$ and $\psi_2:=\pi_2\circ \psi^{-1}$ make the desired diagram commute and both are continuous quotient maps with compact kernels. Uniqueness follows since $\psi$, $\pi_1$, and $\pi_2 $ are unique.
\end{proof}

Theorem~\ref{thm:completion:sandwich} shows all possible completions of a Hecke pair $(G,U)$ differ only by a compact normal subgroup. The locally compact, non-compact structure of a \tdlc completion thus depends only on the Hecke pair; contrast this with the many different profinite completions a group can admit. We give precise statements illustrating this phenomenon in Section~\ref{sec:invariant}.

We conclude this section by making two further observations. First, the Schlichting completion has a natural description. Suppose $(G,U)$ is a Hecke pair and let $\sigma_{(G,U)}:G\rightarrow \Sym(G/U)$ be the permutation representation given by left multiplication. We consider $\Sym(G/U)$ to be equipped with the topology of pointwise convergence.

\begin{prop}\label{prop:sch_description}
For $(G,U)$ a Hecke pair, there is a unique topological group isomorphism $\psi:G\ssl U\rightarrow \ol{\sigma_{(G,U)}(G)}$ such that $\sigma_{(G,U)}=\psi\circ \beta_{G/U}$. 
\end{prop}
\begin{proof}
For $Y\subseteq G$, set $\wh{Y}:=\ol{\sigma_{(G,U)}(Y)}$. The orbits of $\sigma(U)$ are finite on $G/U$, so $\wh{U}$ is a profinite group. On the other hand, $\wh{U}=\Stab_{\wh{G}}(U)$, hence it is open. It now follows that $\wh{G}$ is a \tdlc completion of $G$. 

A basis for the topology on $\wh{G}$ is given by stabilizers of finite sets of cosets. Such stabilizers are exactly of the form $\bigcap_{g\in F}\sigma(g)\wh{U}\sigma(g^{-1})$ with $F\subseteq G$ finite. For every $V\in \U(\wh{G})$, the subgroup $\sigma_{(G,U)}^{-1}(V)$ therefore contains $\bigcap_{g\in F}gUg^{-1}$ for some $F\subseteq G$ finite. The $G$-stable local filter $\mc{S}:=\{\sigma^{-1}_{(G,U)}(V)\mid V\in \U(\wh{G})\}$ is thus exactly the Schlichting filter $\mc{S}_{G/U}$. Theorem~\ref{thm:completion:characterization} now implies the proposition.
\end{proof}

Via Theorem~\ref{thm:completion:sandwich} and Proposition~\ref{prop:sch_description}, we recover a result from the literature.

\begin{cor}[{Shalom--Willis \cite[Lemma~3.5 and Corollary~3.7]{SW_13}}]\label{Schlichting:universal}
Suppose that $G$ is a \tdlc group and that $\phi: G \rightarrow H$ is a completion map. If $U$ is the preimage of a compact open subgroup of $H$, then there is a unique quotient map $\psi: H \rightarrow G \ssl U$ and closed embedding $\iota: G \ssl U \rightarrow \Sym(G/U)$ such that $\sigma_{(G,U)} = \iota \circ \psi \circ \phi$.
\end{cor}

We next show the Belyaev completion satisfies a stronger universality property.

\begin{thm}\label{Belyaev:universal}
Let $\phi: G \rightarrow H$ be a continuous homomorphism such that $H$ is a \tdlc group.  Suppose that $U$ is a commensurated open subgroup of $G$ such that $\ol{\phi(U)}$ is profinite.  Then there is a unique continuous homomorphism $\psi: \hat{G}_U \rightarrow H$ such that $\phi = \psi \circ \beta_{U}$.  If in addition $\phi$ has dense image and $\ol{\phi(U)}$ is open in $H$, then $\psi$ is a quotient map.
\end{thm}

\begin{proof}
Let $\beta := \beta_{U}$ and set $L: = \ol{\phi(U)}$.  Let $V$ be an open subgroup of $H$.  Then $V\cap L$ is open and of finite index in $L$, since $L$ is compact.  In particular, we see that $W = \phi\inv(V) \cap U$ is an open subgroup of $U$ of finite index.  For all open subgroups $V$ of $H$, there is thus $W \in [U]$ such that $\phi(W) \le V$.  We now obtain the unique continuous homomorphism $\psi: \hat{G}_U \rightarrow H$ such that $\phi = \psi \circ \beta$ via Proposition~\ref{prop:filter_image}.

Now suppose that $\ol{\phi(U)}$ is open in $H$ and that $\phi$ has dense image.  The group $\hat{U}:= \ol{\beta(U)}$ is a compact open subgroup of $\hat{G}_U$, so $\psi(\hat{U})$ is compact in $H$; in particular, $\psi(\hat{U})$ is closed in $H$.  We see that $\psi(\hat{U}) \ge \phi(U)$, so  $\psi(\hat{U})$ is indeed open in $H$.  The image of $\psi$ is therefore an open and dense subgroup of $H$, so $\psi$ is surjective.  Since $H$ is a Baire space, it follows that $\psi$ is a quotient map.
\end{proof}

A standard universal property argument now shows that Theorem~\ref{Belyaev:universal} characterizes the Belyaev completion up to topological isomorphism, so one can take Theorem~\ref{Belyaev:universal} as the definition of the Belyaev completion.

\begin{rmk}
We see that the problem of classifying all continuous homomorphisms with dense image from a specified group $G$ (possibly discrete) to an arbitrary \tdlc group can, in principle, be broken into two steps:
\begin{enumerate}[(1)]
\item Classify the possible sizes of completion; in other words, classify the commensurability classes $\alpha = [U]$ of open subgroups of $G$ that are invariant under conjugation.  (This typically amounts to classifying commensurated open subgroups.)

\item For each such class $\alpha$, form the Belyaev completion $\hat{G}_\alpha$ and classify the quotients of $\hat{G}_\alpha$ with compact kernel.
\end{enumerate}

A number of researchers have already considered $(1)$.  Shalom and Willis classify commensurated subgroups of many arithmetic groups in \cite{SW_13}. Other examples include classifications of commensurated subgroups of almost automorphism groups (\cite{LBW15}) and Burger-Mozes simple universal groups (\cite{LW15}).
\end{rmk}

\section{Invariant properties of completions}\label{sec:invariant}
By Theorem~\ref{thm:completion:sandwich}, the \tdlc completions of a group $G$ of a given size differ only by a compact normal subgroup, so ought to have the same ``non-compact" properties. We here make several precise statements showing that this intuition has substance.

\subsection{First invariant properties}

\begin{prop}\label{prop:top_properties}
Let $G$ be a group and let $H$ be a \tdlc completion of $G$ of size $\alpha$. Then,
\begin{enumerate}[(1)] 
\item $H$ is $\sigma$-compact if and only if $|G:W|$ is countable for some (equivalently, any) $W\in \alpha$; 
\item $H$ is compactly generated if and only if $G$ is generated by finitely many left cosets of $W$ for some (equivalently, any) $W\in \alpha $.
\end{enumerate}
\end{prop}

\begin{proof}
Let $\beta:G\rightarrow H$ be the completion map and $V\in \U(H)$.  By hypothesis, $U := \beta\inv(V)$ is an element of $\alpha$.  

For $(1)$, if $H$ is $\sigma$-compact, then $V$ has only countably many left cosets in $H$. Since $U = \beta\inv(V)$, it follows that $|G:U|$ is countable.  Conversely, suppose that $|G:W|$ is countable for some $W\in \alpha$. It follows that $|G:U|$ is countable. Since $\beta(G)$ is dense in $H$, there are only countably many left cosets of $\ol{\beta(U)}$ in $H$, so $H$ is $\sigma$-compact.

For $(2)$, if $H$ is compactly generated, then since $\beta(G)$ is dense in $H$, there exists a finite symmetric $A\subseteq \beta(G)$ such that $H = \grp{A}V$; see for instance \cite[Proposition~2.4]{W_1_14}.  Say $A = \beta(B)$ for a finite subset $B$ of $G$.  For every $g\in G$, there thus exists $v \in V$ and $g'\in \grp{B}$ such that $\beta(g)=\beta(g')v$. Since $\beta\inv(V)=U$, it follows further that $v=\beta(u)$ for some $u\in U$.  Thus, $\beta(\grp{B,U}) = \beta(G)$, and as $\ker\beta \le U$, we infer that $G = \grp{B,U}$.  In particular, $G$ is generated by finitely many left cosets of $U$.  Conversely, suppose that $G$ is generated by finitely many left cosets of some $W\in \alpha$. It follows that $G$ is generated by finitely many left cosets $b_1U, b_2U, \dots, b_nU$ of $U$. The image $\beta(\bigcup^n_{i=1}b_iU)$ generates a dense subgroup of $H$, and hence the compact subset $X: = \bigcup^n_{i=1}\beta(b_i)\ol{\beta(U)}$ generates a dense open subgroup of $H$ and therefore generates $H$.
\end{proof}

\begin{cor}\label{cor:top_properties}
For each of the following properties, either every completion of a group $G$ of size $\alpha$ has the property, or every completion of $G$ of size $\alpha$ fails to have the property.
\begin{enumerate}[(1)]
\item Being $\sigma$-compact.
\item Being compactly generated.
\end{enumerate}
\end{cor}

We next consider possible quotient groups and amenability.
\begin{prop}\label{prop:completion:invariants}
For each of the following properties, either every completion of a group $G$ of size $\alpha$ has the property, or every completion of $G$ of size $\alpha$ fails to have the property.
\begin{enumerate}[(1)]
\item Having a quotient isomorphic to $N$ where $N$ is any specified \tdlc group that has no non-trivial compact normal subgroups.
\item Being amenable.
\end{enumerate}
\end{prop}

\begin{proof}
For $i\in \{1,2\}$, let $\phi_i: G \rightarrow H_i$ be a completion map of size $\alpha$ and let $U_i$ be the preimage under $\phi_i$ of a compact open subgroup of $H_i$.  The subgroup $U_i$ is a member of $\alpha$, so by Theorem~\ref{thm:completion:sandwich}, there are quotient maps $\pi_i: \hat{G}_U \rightarrow H_i$ with compact kernel.  It therefore suffices to show that for any \tdlc group $H$ and compact normal subgroup $K$ of $H$, the group $H$ has the property if and only if $H/K$ does.

For $(1)$, if $\pi: H \rightarrow N$ is a quotient map, then $\pi(K)$ is a compact normal subgroup of $N$. Since $N$ has no non-trivial compact normal subgroup, we deduce that $K \le \ker\pi$, so $N$ is a quotient of $H$ if and only if $N$ is a quotient of $H/K$. 

For $(2)$, recall that every compact subgroup is amenable and that if $L$ is a closed normal subgroup of the locally compact group $H$, then $H$ is amenable if and only if both $H/L$ and $L$ are amenable.  Since $K$ is compact, we deduce that $H$ is amenable if and only if $H/K$ is amenable.
\end{proof}

Let us now consider topological countability axioms. These are more delicate and depend on the choice of $G$-stable local filter, as opposed to only the size.

\begin{prop}\label{prop:first_countable} 
Suppose $G$ is a topological group and $\mc{S}$ is a $G$-stable local filter. Then
$\hat{G}_{\mc{S}}$ is first countable if and only if the set $\{V \in \mc{S} \mid V \le U\}$ is countable, for some (equivalently, any) $U \in \mc{S}$.
\end{prop}
\begin{proof}
Fix $U \in \mc{S}$.  Let $\beta: G \rightarrow \hat{G}_{\mc{S}}$ be the completion map and set $\mc{S}_U := \{V \in \mc{S} \mid V \le U\}$.  By Theorem~\ref{thm:completion:def}, we have $|\mc{S}_U|= |\mc{O}|$ where $\mc{O}$ is the set of open subgroups of $U$.  If $\mc{S}_U$ is countable, then $\mc{O}$ is a countable base of identity neighborhoods, so $\hat{G}_{\mc{S}}$ is first countable.  Conversely if $\hat{G}_{\mc{S}}$ is first countable, then there is a countable base $\mc{B}$ of identity neighborhoods consisting of open subsets of $U$.  Since $U$ is profinite, each $B \in \mc{B}$ contains a subgroup of $U$ of finite index, so there are only finitely many open subgroups of $U$ that contain $B$.  Hence, $\mc{O}$ is countable, implying that $\mc{S}_U$ is countable.
\end{proof}

\begin{prop}\label{prop:tdlcsc_completion}
Let $G$ be a topological group and $\mc{S}$ be a $G$-stable local filter.  Then $\hat{G}_{\mc{S}}$ is second countable if and only if $\{V \in \mc{S} \mid V \le U\}$ and $|G:U|$ are countable, for some (equivalently, any) $U \in \mc{S}$.
\end{prop}
\begin{proof}
Via \cite[(5.3)]{K95}, a locally compact group is second countable if and only if it is $\sigma$-compact and first countable. The proposition now follows from Propositions~\ref{prop:first_countable} and \ref{prop:top_properties}.
\end{proof}

\begin{cor}\label{cor:Sch_tdlcsc} 
If $(G,U)$ is a Hecke pair such that $|G:U|$ is countable, then the Schlichting completion $G\ssl U$ is a \tdlcsc group.
\end{cor}

\subsection{Elementary groups}
Let us now consider a more complicated algebraic property, the property of being an elementary group. 

\begin{defn}\label{def:elementary}The class $\Es$ of \textbf{elementary groups} is the smallest class of \tdlcsc groups such that
	\begin{enumerate}[(i)]
		\item $\Es$ contains all second countable profinite groups and countable discrete groups.
		\item $\Es$ is closed under taking closed subgroups.
		\item $\Es$ is closed under taking Hausdorff quotients.
		\item $\Es$ is closed under forming group extensions.
		\item If $G$ is a \tdlcsc group and $G=\bigcup_{i\in \Nb}O_i$ where $(O_i)_{i\in \Nb}$ is an $\subseteq$-increasing sequence of open subgroups of $G$ with $O_i\in\Es$ for each $i$, then $G\in\Es$. We say that $\Es$ is \textbf{closed under countable increasing unions}.
	\end{enumerate}
\end{defn} 
The operations $\mathrm{(ii)}$-$\mathrm{(v)}$ are often called the \textbf{elementary operations}. It turns out operations $\mathrm{(ii)}$ and $\mathrm{(iii)}$ follow from the others, and $\mathrm{(iv)}$ can be weakened to $\mathrm{(iv)}'$: \textit{$\Es$ is closed under extensions of profinite groups and discrete groups}. These results are given by \cite[Theorem 1.3]{W_1_14}. 

\begin{rmk} If $G$ is a \tdlcsc group that is non-discrete, compactly generated, and topologically simple, then $G$ is \emph{not} a member of $\Es$.  The class $\Es$ is thus strictly smaller than the class of all \tdlcsc groups.
\end{rmk}

The class of elementary groups comes with a canonical \textit{successor ordinal} valued rank called the \defbold{decomposition rank} and denoted by  $\xi(G)$; see \cite[Section 4]{W_1_14}. The key property of the decomposition rank that we will exploit herein is that it is well-behaved under applying natural group building operations. For a \tdlc group $G$, recall that the \defbold{discrete residual} of $G$ is defined to be $\Res{}(G):=\bigcap \{O\normal G\mid O\text{ open}\}$.

\begin{prop}\label{prop:xi_monotone}
For $G$ a non-trivial elementary group, the following hold.
\begin{enumerate}[(1)]
 \item If $H$ is a \tdlcsc group, and $\psi:H\rightarrow G$ is a continuous, injective homomorphism, then $H$ is elementary with $\xi(H)\leq \xi(G)$. \textup{(\cite[Corollary 4.10]{W_1_14})}
 
\item If $L\normal G$ is closed, then $\xi(G/L)\leq \xi(G)$. \textup {(\cite[Theorem 4.19]{W_1_14})}

\item If $G=\bigcup_{i\in \Nb}O_i$ with $(O_i)_{i\in \Nb}$ an $\subseteq$-increasing sequence of compactly generated open subgroups of $G$, then $\xi(G)=\sup_{i\in \Nb}\xi(\Res{}(O_i))+1$. If $G$ is compactly generated, then $\xi(G)=\xi(\Res{}(G))+1$. \textup{(\cite[Lemma 4.12]{W_1_14})}

\item If $G$ is residually discrete, then $G$ is elementary with $\xi(G)\leq 2$. \textup{(\cite[Observation 4.11]{W_1_14})}

\item If $G$ is elementary and lies in a short exact sequence of topological groups
\[
\triv\rightarrow N\rightarrow G\rightarrow Q\rightarrow \triv,
\]
then $\xi(G)\leq (\xi(N)-1)+\xi(Q)$.  (Here $(\xi(N) - 1)$ denotes the predecessor of $\xi(N)$, which exists as $\xi(N)$ is a successor ordinal.)  \textup{(\cite[Lemma 3.8]{RW_DenseLC_15})}
\end{enumerate}
\end{prop}

\begin{prop}\label{prop:completion:elementary_rank}
Let $G$ be a group.
\begin{enumerate}[(1)]
\item Either every second countable completion of $G$ of size $\alpha$ is elementary, or all second countable completions of $G$ of size $\alpha$ are non-elementary.

\item If every second countable completion of size $\alpha$ is elementary, then for any two second countable completions $H$ and $L$ of $G$ with size $\alpha$, we have
\[
\xi(L)\leq 1+\xi(H).
\]

\item If some second countable completion of size $\alpha$ is elementary with transfinite rank $\beta$, then every second countable completion of size $\alpha$ is elementary with rank $\beta$.
\end{enumerate}
\end{prop}

\begin{proof}
Let $H$ and $L$ be second countable completions of $G$ of size $\alpha$.  By Theorem~\ref{thm:completion:characterization} we may assume $H = \hat{G}_{\mc{S}_1}$ and $L = \hat{G}_{\mc{S}_2}$ where $\mc{S}_1, \mc{S}_2 \subseteq \alpha$ are $G$-stable local filters.  Let $\mc{S}_3$ be the smallest $G$-stable local filter containing $\mc{S}_1 \cup \mc{S}_2$.  Via Proposition~\ref{prop:tdlcsc_completion}, $\{V \in \mc{S}_i \mid V \le U\}$  is countable for $i \in \{1,2\}$ and $U \in \mc{S}_i$. It follows that  $\{V \in \mc{S}_3 \mid V \le U\}$ is countable for $U \in \mc{S}_1$, and hence $\hat{G}_{\mc{S}_3}$ is first countable.  By Corollary~\ref{cor:top_properties}(1), $\hat{G}_{\mc{S}_3}$ is also $\sigma$-compact and hence second countable.

By Proposition~\ref{prop:completions:comparison}, both $\hat{G}_{\mc{S}_1}$ and $\hat{G}_{\mc{S}_2}$ are quotients of $\hat{G}_{\mc{S}_3}$ with compact kernel.  It follows that $\hat{G}_{\mc{S}_i}$ is elementary for $i \in \{1,2\}$ if and only if $\hat{G}_{\mc{S}_3}$ is elementary; in particular, it is not possible for $\hat{G}_{\mc{S}_1}$ to be elementary and $\hat{G}_{\mc{S}_2}$ to be non-elementary.  This proves $(1)$.  Moreover, if $\hat{G}_{\mc{S}_3}$ is elementary, then via Proposition~\ref{prop:xi_monotone} we obtain the inequalities
\[
\xi(\hat{G}_{\mc{S}_i}) \le \xi(\hat{G}_{\mc{S}_3}) \le  1+ \xi(\hat{G}_{\mc{S}_i}) \quad (i \in \{1,2\}).
\]
In particular, if $\xi(\hat{G}_{\mc{S}_3})$ is finite then $\xi(\hat{G}_{\mc{S}_1}),\xi(\hat{G}_{\mc{S}_2}) \in \{\xi(\hat{G}_{\mc{S}_3}) - 1,\xi(\hat{G}_{\mc{S}_3})\}$, and if $\xi(\hat{G}_{\mc{S}_3})$ is transfinite then it is equal to both $\xi(\hat{G}_{\mc{S}_1})$ and $\xi(\hat{G}_{\mc{S}_2})$.  This proves $(2)$ and $(3)$.
\end{proof}

It can be the case that there are \textit{no} second countable completions of a given size. In light of Corollary~\ref{cor:Sch_tdlcsc}, however, a group $G$ has a second countable completion of size $\alpha$ if and only if $|G:U|$ is countable for some $U\in \alpha$. With this in mind, we make the following definition.

\begin{defn} 
A size $\alpha$ of a group $G$ is called \textbf{elementary} if there is $U\in \alpha$ such that $|G:U|$ is countable and some (all) second countable completions are elementary. A Hecke pair $(G,U)$ is called \defbold{elementary} if $|G:U|$ is countable and some (all) \scou completions are elementary.
\end{defn}

\subsection{The scale function and flat subgroups}
We conclude this section by considering the scale function and flat subgroups in relation to completions; these concepts were introduced in \cite{Will94} and \cite{WillisFlat} respectively, although the term ``flat subgroup" is more recent.

\begin{defn}
For $G$ a \tdlc group, the \defbold{scale function} $s:G\rightarrow \Zb$ is defined by
\[
s(g):=\min\{|gUg\inv:gUg\inv \cap U|\mid U\in \U(G)\}.
\]
A compact open subgroup $U$ of $G$ is \defbold{tidy} for $g\in G$ if it achieves $s(g)$.  We say $g$ is \defbold{uniscalar} if $s(g) = s(g\inv) = 1$.

A subset $X$ of $G$ is \defbold{flat} if there exists a compact open subgroup $U$ of $G$ such that for all $x \in X$, the subgroup $U$ is tidy for $x$; in this case, we say $U$ is tidy for $X$. If $X$ is a finitely generated flat subgroup, the \defbold{rank} of $X$ is the least number of generators for the quotient group $X/\{x\in X\mid s(x)=1\}$.
\end{defn}

The scale function and flatness are clearly locally compact non-compact phenomena. In relation to \tdlc completions, they only depend on the size $[U]$.

\begin{prop}\label{prop:completion:tidy}
For $\phi: G \rightarrow H$ a \tdlc completion of size $[U]$, the following hold:
\begin{enumerate}[(1)]
\item For $\hat{s}$ and $s$ the scale functions for $\hat{G}_U$ and $H$, $s\circ \phi=\hat{s}\circ \beta_U$. 

\item For $X\subseteq G$, the subset $\phi(X)$ is flat if and only if $\beta_U(X)$ is flat.

\item If $K\leq G$ is a finitely generated subgroup, then $\phi(K)$ is flat with rank $k$ if and only if $\beta_U(K)$ is flat with rank $k$.
\end{enumerate}
\end{prop}
\begin{proof}
By Theorem~\ref{thm:completion:sandwich}, we can factorize $\phi$ as $\phi =  \pi\circ \beta_U$, where $\pi:\hat{G}_U\rightarrow H$ is a quotient map with compact kernel. The result \cite[Lemma~4.9]{ReidFlat} ensures that $s \circ \pi = \hat{s}$, hence $s \circ \phi = \hat{s} \circ \beta_U$, proving $(1)$.

Appealing again to \cite[Lemma~4.9]{ReidFlat}, if $U$ is tidy for $g$ in $\hat{G}_{U}$, then $\pi(U)$ is tidy for $\pi(g)$ in $H$, and conversely if $V$ is tidy for $\pi(g)$ in $H$, then $\pi\inv(V)$ is tidy for $g$ in $\hat{G}_{U}$.  Therefore, if $\beta_U(X)$ has a common tidy subgroup, then so does $\phi(X)$.  Conversely, if $\phi(X)$ has a common tidy subgroup $V$ in $H$, then $\pi\inv(V)$ is a common tidy subgroup for $\beta_U(X)$.  We conclude that $\phi(X)$ has a common tidy subgroup if and only if $\beta_U(X)$ does, verifying $(2)$.

Finally, if $K$ is a subgroup of $G$, then $\phi(K)$ is a flat subgroup of $H$ if and only if $\beta_U(K)$ is a flat subgroup of $\hat{G}_{U}$ by $(2)$.  The rank of $\phi(K)$ is the number of generators of the factor $\phi(K)/L_{H}$ where $L_H := \{x\in \phi(K)\mid s(x)=1\}$. Letting $L_{\hat{G}_U}$ be the analogous subgroup of $\hat{G}_U$, it follows from $(1)$ that the map $\pi$ induces an isomorphism $\tilde{\pi}:\beta_U(K)/L_{\hat{G}_U}\rightarrow \phi(K)/L_H$. We conclude that $\beta_U(K)$ has rank $k$ if and only if $\phi(K)$ has rank $k$, proving $(3)$.
\end{proof}

The next corollary is immediate from Proposition~\ref{prop:completion:tidy} and the fact the scale function is continuous (see \cite{Will94}).
\begin{cor} 
For $G$ a group and $U$ a commensurated open subgroup of $G$, either all \tdlc completions of $G$ of size $\alpha$ are uniscalar, or no completion  of size $\alpha$ is uniscalar.
\end{cor}

\section{Completions compatible with homomorphisms}

For an injective homomorphism $\theta: G \rightarrow L$, we may wish to find a \tdlc completion $\widetilde{G}$ of $G$ such that $\theta$ extends to an injective homomorphism from $\widetilde{G}$ to $L$.  More precisely, we say the \tdlc completion map $\beta: G \rightarrow \widetilde{G}$ is \defbold{compatible with $\theta$} if there is a continuous injective homomorphism $\psi: \widetilde{G} \rightarrow L$ such that $\theta = \psi \circ \beta$; note that in this case $\psi$ is necessarily unique.  Here we do not insist that $L$ be locally compact; indeed, in many interesting examples, $L$ itself will not be locally compact (see Remark~\ref{rmk:Polish} below).  We can characterize the \tdlc completions compatible with $\theta$ in terms of commensurated subgroups.

\begin{thm}\label{thm:compatible_completion}
Let $\theta: G \rightarrow L$ be an injective continuous homomorphism of topological groups.

\begin{enumerate}[(1)]
\item Suppose that $H$ is an open commensurated subgroup of $G$ such that the closure $\ol{\theta(H)}$ of $\theta(H)$ in $L$ is profinite and set $H^* := \theta\inv(\ol{\theta(H)})$.  Then $H^*$ is open and commensurated in $G$, and there is a \tdlc completion map $\beta: G \rightarrow \hat{G}_{H,\theta}$ compatible with $\theta$ such that $H^*$ is the preimage of a compact open subgroup of $\hat{G}_{H,\theta}$.  Moreover, $\beta$ is unique up to isomorphisms of $\hat{G}_{H,\theta}$ and is determined by the pair $([H^*],\theta)$.

\item Suppose that $\beta: G \rightarrow \widetilde{G}$ is a \tdlc completion of $G$ compatible with $\theta$ and $\psi:\widetilde{G}\rightarrow L$ is such that $\theta=\psi\circ \beta$, let $U$ be a compact open subgroup of $\widetilde{G}$, and set $H: = \beta\inv(U)$.  Then $H$ is a commensurated subgroup of $G$ that is the preimage of a profinite subgroup $\ol{\theta(H)} = \psi(U)$ of $L$, and $\widetilde{G} \simeq \hat{G}_{H,\theta}$.
\end{enumerate}
\end{thm}

\begin{proof}
For (1), $H$ is an open commensurated subgroup of $G$ such that the closure $K = \ol{\theta(H)}$ of $\theta(H)$ in $L$ is profinite. 
The image $\theta(G)$ additionally commensurates $K$; consider, for instance, \cite[Lemma 2.7]{LBW15}. We thus conclude that $H^* := \theta\inv(K)$ is commensurated in $G$.  Now let $\mc{R}$ be the set of closed subgroups of $L$ that are commensurate with $K$ and let $\mc{S}$ be the set of $\theta$-preimages of elements of $\mc{R}$. The collection $\mc{S}$ forms a $G$-stable local filter, and setting $\hat{G}_{H,\theta} := \hat{G}_{\mc{S}}$, we obtain a \tdlc completion $\beta: G \rightarrow \hat{G}_{H,\theta}$.

Let $L'$ be the group $\langle \theta(G),K\rangle$, equipped with the unique group topology such that the inclusion of $K$ into $L'$ is continuous and open.  The map $\theta$ induces a continuous homomorphism $\theta'$ from $G$ to $L'$, and Theorem~\ref{thm:completion:characterization} provides a unique topological group isomorphism $\psi': \hat{G}_{H,\theta} \rightarrow L'$ such that $\theta' = \psi' \circ \beta$.  In particular, since the natural inclusion of $L'$ into $L$ is continuous, we obtain a continuous injective homomorphism $\psi:\hat{G}_{H,\theta} \rightarrow L$ such that $\theta = \psi \circ \beta$.  Thus $\beta$ is compatible with $\theta$.  It is also clear that given $\theta$, the construction of $\beta$ is determined by the commensurability class of $K$ among closed subgroups of $L$, and hence by the commensurability class of $H^*$, since $K = \ol{\theta(H^*)}$ and the mapping $\cdot \mapsto \ol{\theta(\cdot)}$ preserves commensurability classes of subgroups.

To see that $\beta$ is unique up to isomorphisms of the range, Theorem~\ref{thm:completion:characterization} ensures that it is enough to show the following: given a \tdlc completion $\beta_{(G,\mc{T})}: G \rightarrow \hat{G}_{\mc{T}}$ that is compatible with $\theta$, where $\mc{T}$ is a $G$-stable local filter and $H^* \in \mc{T}$, then $\mc{T} = \mc{S}$.  Suppose $\psi_2$ is the injective continuous homomorphism from $\hat{G}_{\mc{T}}$ to $L$ such that $\theta = \psi_2 \circ \beta_{(G,\mc{T})}$.  The collection $\mc{T}$ is the set of $\beta_{(G,\mc{T})}$-preimages of compact open subgroups of $\hat{G}_{\mc{T}}$. The images of the compact open subgroups of $\hat{G}_{\mc{T}}$ give rise to a collection $\mc{R}'$ of compact subgroups of $L$, so $\mc{T}$ is the set of $\theta$-preimages of elements of $\mc{R}'$.  We see that $K \in \mc{R'}$ and that all elements of $\mc{R}'$ are commensurate with $K$, so $\mc{R}' \subseteq \mc{R}$ and hence $\mc{T} \subseteq \mc{S}$.  The argument that $\mc{S} \subseteq \mc{T}$ is similar.

\medskip

For (2), $U$ is commensurated in $\widetilde{G}$, so $H$ is commensurated in $G$.  Let $K := \psi(U)$.  Since $\psi$ is injective and continuous and $U$ is compact, we see that $K$ is closed in $L$ and isomorphic to $U$ as a topological group; in particular, $K$ is profinite.  The injectivity of $\psi$ ensures that $U = \psi\inv(K)$, so $H = \beta\inv\psi\inv(K) = \theta\inv(K)$.  Since $G$ has dense image in $\widetilde{G}$, we see that $\beta(H)$ is dense in $U$ and hence $\theta(H)$ is dense in $K$.  The fact that $\widetilde{G} \simeq \hat{G}_{H,\theta}$ follows from the uniqueness result established in part (1).
\end{proof}

To illustrate the theorem, let us spell out what it means for certain classes of action.  Given a group $G$ acting faithfully on a structure $X$, a \defbold{\tdlc completion} of the action is a faithful action of a \tdlc group $\widetilde{G}$ on the same structure such that $\widetilde{G}$ contains a dense copy of $G$ with its original action.  Say that a unitary representation $X$ of a group $H$ is \defbold{locally finite} if $X$ is the closure of the union of an increasing family $(X_i)$ of finite-dimensional subrepresentations, such that the kernel of the action of $H$ on $X_i$ has finite index for each $i$.

\begin{cor}\label{cor:action_extension}
For $G$ a group with $H$ a commensurated subgroup, suppose that one of the following hold:
\begin{enumerate}[(1)]
\item $G$ acts faithfully by permutations on a set $X$, and $H$ has finite orbits on $X$.
\item $G$ acts faithfully by homeomorphisms on a compact metrizable zero-dimensional space $X$, and the action of $H$ on $X$ is equicontinuous.
\item $X$ is a faithful complex unitary representation of $G$, and $X$ is locally finite as a representation of $H$.
\end{enumerate}
Then there is a unique \tdlc completion $\widetilde{G} \curvearrowright X$ of the action such that $H$ has compact open closure in $\widetilde{G}$, which is continuous in cases (1) and (2), and strongly continuous in case (3).  Moreover, all \tdlc completions of the action of $G$ on $X$ with the given continuity property arise in this way.
\end{cor}

\begin{proof}
For $(1)$, consider $G$ and $H$ as subgroups of $\Sym(X)$ with the permutation topology.  As is well-known, a subgroup $H$ of $\Sym(X)$ has compact closure if and only if it has finite orbits.  Furthermore, any topological permutation group that acts continuously on $X$ must map continuously into $\Sym(X)$, and conversely, $\Sym(X)$ itself acts continuously on $X$.

For $(2)$, consider $G$ and $H$ as subgroups of $\Homeo(X)$ with the compact-open topology.  By the Arzel\`{a}--Ascoli theorem, the condition that $H$ be equicontinuous on $X$ is exactly the condition that $H$ has compact closure in $\Homeo(X)$.  If $\widetilde{G}$ is a group acting faithfully by homeomorphisms on $X$, then the corresponding homomorphism from $\widetilde{G}$ to $\Homeo(X)$ is continuous if and only if the action of $\widetilde{G}$ on $X$ is continuous.

For $(3)$, consider $G$ and $H$ as subgroups of the unitary group $\mathrm{U}(X)$ with the strong operator topology and let $\ol{H}$ denote the closure of $H$ in $\mathrm{U}(H)$.  Suppose that $X$ is locally finite as a representation of $H$, with $(X_i)$ the corresponding increasing family of finite-dimensional subrepresentations.  For each $i$, $\ol{H}$ acts on $X_i$ with an open kernel of finite index, so $\ol{H}$ is totally disconnected.  In addition, given a net $(h_j)_{j \in J}$ in $\ol{H}$, there is a subnet $(h_{j(k)})$ that is eventually constant on each $X_i$. It follows that $(h_{j(k)})$ converges pointwise on $X$ to a unitary map; in other words, $h_{j(k)}$ converges in $\ol{H}$.  Thus $\ol{H}$ is compact and hence a profinite group.  Conversely, suppose $H$ is a subgroup of $G$ with compact closure in $\mathrm{U}(X)$.  By standard results (see for instance \cite[Theorem~5.2]{Folland}), $X$ is an orthogonal sum of finite-dimensional irreducible representations $X_j$ of $\ol{H}$, and on each $X_j$, $\ol{H}$ acts as a compact Lie group.  If $\ol{H}$ is profinite, then the Lie quotients of $\ol{H}$ are in fact finite, so $H$ acts on $X_j$ with a kernel of finite index.  We conclude that $X$ is a locally finite representation of $\ol{H}$ and hence of $H$.  Therefore, $H$ has profinite closure in $\mathrm{U}(X)$ if and only if $X$ is locally finite as a representation of $H$.

In all cases, the conclusions now follow by Theorem~\ref{thm:compatible_completion}.
\end{proof}

\begin{rmk}\label{rmk:Polish}
If $X$ is a countably infinite set, the Cantor set, or the infinite-dimensional separable complex Hilbert space, then $\Sym(X)$, $\Homeo(X)$, and $\mathrm{U}(X)$ respectively are well-known examples of Polish groups that are not locally compact.  As such, simply taking the closure of the image of $\psi:G\rightarrow L$, with $L$ one of the aforementioned groups,  will not always produce a locally compact group, and moreover, there are interesting examples of continuous actions of \tdlc groups that do not arise from taking the closure in $L$.  For example, Thompson's group $V$ acts faithfully by homemorphisms on the standard ternary Cantor set $X \subset [0,1]$ and has a commensurated subgroup $H$ consisting of those elements of $V$ that act by isometries of the visual metric. In particular, the action of $H$ is equicontinuous. There is thus a unique \tdlc completion $\widetilde{V} \curvearrowright X$ of the action on $X$ such that the closure of $H$ in $\widetilde{V}$ is a compact open subgroup.  The group $\widetilde{V}$ is known as Neretin's group $\mc{N}_{2,2}$ of piecewise homotheties of the ternary Cantor set, and it carries a strictly finer topology than the one induced by $\Homeo(X)$.
\end{rmk}


\bibliographystyle{abbrv}
\bibliography{biblio}{}

\def\cprime{$'$} \def\cprime{$'$}
\begin{thebibliography}{10}

\bibitem{Belyaev}
V.~V. Belyaev.
\newblock Locally finite groups containing a finite inseparable subgroup.
\newblock {\em Siberian Math. J.}, 34(2):218--232, 1993.

\bibitem{B_top1_89}
N.~Bourbaki.
\newblock {\em General topology. {C}hapters 1--4}.
\newblock Elements of Mathematics (Berlin). Springer-Verlag, Berlin, 1989.
\newblock Translated from the French, Reprint of the 1966 edition.

\bibitem{Folland}
G.~B. Folland.
\newblock {\em A course in abstract harmonic analysis}.
\newblock CRC Press, Boca Raton, FL, 2000.

\bibitem{K95}
A.~S. Kechris.
\newblock {\em Classical descriptive set theory}, volume 156 of {\em Graduate
  Texts in Mathematics}.
\newblock Springer-Verlag, New York, 1995.

\bibitem{LBW15}
A.~Le~Boudec and P.~Wesolek.
\newblock Commensurated subgroups in tree almost automorphism groups.
\newblock {\em Groups Geom. Dyn.}
\newblock to appear.

\bibitem{LW15}
F.~Le~Ma\^itre and P.~Wesolek.
\newblock On strongly just infinite profinite branch groups.
\newblock {\em J. Group Theory}, 20(1):1--32, 2017.

\bibitem{ReidFlat}
C.~D. Reid.
\newblock Dynamics of flat actions on totally disconnected, locally compact
  groups.
\newblock {\em New York J. Math.}, 22:115--190, 2016.

\bibitem{RW_DenseLC_15}
C.~D. Reid and P.~R. Wesolek.
\newblock Dense normal subgroups and chief factors in locally compact groups.
\newblock {\em Proc. Lond. Math. Soc.}
\newblock to appear.

\bibitem{RZ00}
L.~Ribes and P.~Zalesskii.
\newblock {\em Profinite groups}, volume~40 of {\em Ergebnisse der Mathematik
  und ihrer Grenzgebiete. 3. Folge. A Series of Modern Surveys in Mathematics
  [Results in Mathematics and Related Areas. 3rd Series. A Series of Modern
  Surveys in Mathematics]}.
\newblock Springer-Verlag, Berlin, second edition, 2010.

\bibitem{Schlichting}
G.~Schlichting.
\newblock Operationen mit periodischen {S}tabilisatoren.
\newblock {\em Arch. Math.}, 34(1):97--99, 1980.

\bibitem{SW_13}
Y.~Shalom and G.~A. Willis.
\newblock Commensurated subgroups of arithmetic groups, totally disconnected
  groups and adelic rigidity.
\newblock {\em Geom. Funct. Anal.}, 23(5):1631--1683, 2013.

\bibitem{W_1_14}
P.~Wesolek.
\newblock Elementary totally disconnected locally compact groups.
\newblock {\em Proc. Lond. Math. Soc. (3)}, 110(6):1387--1434, 2015.

\bibitem{Will94}
G.~Willis.
\newblock The structure of totally disconnected, locally compact groups.
\newblock {\em Math. Ann.}, 300(2):341--363, 1994.

\bibitem{WillisFlat}
G.~Willis.
\newblock Tidy subgroups for commuting automorphisms of totally disconnected
  groups: An analogue of simultaneous triangularisation of matrices.
\newblock {\em New York J. Math.}, 10:1--35, 2004.

\end{thebibliography}

\end{document}